\documentclass[12pt]{article}
\usepackage{amsmath,amssymb,amsfonts,setspace}
\usepackage{mathtools}

\newtheorem{Theorem}{Theorem}

\newtheorem{Lemma}{Lemma}
\newtheorem{corollary}{Corollary}

\begin{document}

\title{Infinite Size-Biased Orders}

\author{Alexander Gnedin \\ Queen Mary, University of London}

\maketitle

\begin{abstract}
\noindent
The infinite random size-biased order with arbitrary positive size parameters is introduced in terms of independent exponential random variables.
We collect basic properties and constructions of the order, some of which belong to the folklore, and show how the order  type (e.g. ${\mathbb Z}_{>0}, {\mathbb Q}$ or any other possible)
depends on parameters.
\end{abstract}

\section{Introduction}

Permutations of finite sets  are particularly exciting combinatorial structures due to interplay between their roles as bijections  and (linear) orders.
Though these are inseparable,  from the time of Cauchy the early literature 
was careful  to distinguish the term `permutation' from `substitution',  the latter being understood as the operation of replacing each element of the set by some other.
For infinite sets  the hypostases truly disentangle. 
The self-bijections  of ${\mathbb Z}_{>0}$, introduced by Vitali \cite{Vitali} as infinite substitutions,  
correspond to the orders obtained by re-arranging the positive integers in a sequence. In this paper under infinite permutation we shall mean this sort of order, corresponding to Vitali's substitution;
see \cite{FDF} for a wider use of the term.
But such infinite permutations constitute just one isomorphism type 
from  the continuum  universe   of  types of countable (infinite) orders \cite{Rosenstein}.
Recall that an order is {\it dense} if for any two distinct elements there is a third element strictly between them.
By Cantor's isomorphism theorem every countable dense order without  first and  last  elements is of the type ${\mathbb Q}$, while any other type can be realised 
as a suborder of the rationals.
For many reasons, however, it is more natural  to consider countable orders embedded in reals ,
though a dense order may be realised by a set not dense in the topological sense.

For random order chosen from some probability distribution over the space of orders on 
a given countable  set, the type 
may be a nontrivial random variable. We say that a random order is {\it pure} if  it has one type almost surely.
For instance, pure orders of   types ${\mathbb Z}_{>0}$ and ${\mathbb Z}$ have been
introduced recently as extensions of Mallows distribution on finite permutations \cite{GO1, GO2}.
To compare,
the analogue of finite uniformly distributed permutation is the {\it exchangeable} order, which has the type 
${\mathbb Q}$ and is derived  from the natural order on the values sampled in the i.i.d. fashion  
from a continuous distribution over reals.
The idea of the latter construction was just sketched in an abstract by  Rubin \cite{Rubin} and employed later as a model for the  `secretary' problem with infinitely many
choice options \cite{Gianini}.

We will focus on a rich infinite-parameter family of random orders whose type varies with parameters.
A good starting point is the following explicit construction in the spirit of finite `order statistics' model of ranking \cite{Ranking}.
Let $I$ be a countable   set of items equipped with a {\it size} function $w(i)>0$. For independent exponential random variables $X_i, i\in I,$ with $X_i$ having rate $w(i)$, define a random 
relation $\triangleleft$ on $I$ by setting
\begin{equation}\label{exp-r}
i\triangleleft j \Longleftrightarrow X_i<X_j,~~~i\neq j.
\end{equation}
With probability one the values of $X_i$'s are all pairwise distinct, therefore (\ref{exp-r}) defines a strict linear order,
 modulo a null event.
 We call $\triangleleft$ a {\it size-biased order} on $I$. The probability of  relation $i\triangleleft j$  is $w(i)/(w(i)+w(j))$, which exceeds $1/2$ whenever  $w(i)>w(j)$, so 
the intuitive effect of size-biasing is that the items of bigger size are more likely to precede those with smaller.

Finite  orders of the kind 
emerge from the statistical procedure of size-biased sampling without replacement from a  finite stratified population.
Yates and Grundy \cite{Yates} introduced the method as follows: select the first unit [stratum, item] with probability proportional to size, and the second with probability proportional to the size of the remaining units, and so on.  The list of size-biased picks makes up what is now called size-biased permutation.
Another eminent origin of the concept is the Luce model of choice \cite{Luce},  which postulates a condition on the probability that an item is recognised as
`the best' within a set of choosable alternatives.  
Yellott \cite{Yellott} explored the connection between Luce's axiomatic approach and the representation of choice 
via independent random variables from the Gumbel 
translation family (equivalent  to (\ref{exp-r})  by passing to
$-\log X_i$).
Fishburn \cite{Fishburn} studied conditions on random rankings compatible with Luce's axiom.
Steele \cite{Steele} identified all possible limits of empirical distributions contructed from the choice probabilities.
Many results related to the finite size-biased orders appeared (sometimes in disguise)  in the work  on nonuniform permutation models \cite{Plackett},   reliability \cite{Tikhov},  nonstationary models for records \cite{Borovkov,  Nevzorov, Pfeifer, Yang} and exponential order statistics \cite{Dansie,  Nagaraja, Nevzorov}. See \cite{Diaconis, Fill, PitmanTran} for different aspects and further pointers to the extensive literature.

For countable ground set $I$ and  summable size function, i.e.  satisfying $\sum_{i\in I}w(i)<\infty$,  constructing the order by size-biased sampling 
works without change. 
The term `size-biased permutation'  was first applied to this infinite case   in a study of ecological diversity \cite{Patil}. 
We can indeed speak of a permutation,
 because eventually every element gets sampled.
Another interpretation of $\triangleleft$  as `age ordering'
derives from sampling {\it with} replacement from the discrete probability distribution  with  masses 
$w(i)/\sum_{j\in I}w(j), ~i\in I$. 
By the i.i.d. sampling, where repetitions are inevitable, the order $\triangleleft$ corresponds to
 the random succession  of distinct items listed as they appear for the first time in the sample.
  The induced ordering
of $w(i)$'s (which need not be strict)  is also called size-biased;
this  offers a useful alternative to the arrangement of the collection of 
sizes by decrease, especially  in the analysis of discrete distributions with random masses such as the Poisson-Dirichlet random measures and their relatives \cite{Feng, G-SBP, PPY}.

The infinite exchangeable order appears in the special case of  constant size, e.g. $w(i)\equiv 1$.
This and other  size-biased orders with nonsummable size functions provide a natural framework to model processes of records.  
Suppose $I={\mathbb Z}_{>0}$, so the items are arranged in a sequence. In the event 
$A_i:=\{X_i=\min (X_1,\ldots, X_i)\}$ we speak of a record at index $i$.
The  setting (\ref{exp-r}) of exponential variables  is an instance of Nevzorov's model for records, which has a characteristic feature 
that the  $A_i$'s are independent   \cite{Borovkov, Nevzorov}
\footnote{Our records are {\it lower} (or minimal). To literally fit in Nevzorov's $F^\alpha$ scheme for {\it upper} records we need
to consider the negatives $-X_i$.}.
The summable case has only limited interest for the theory, because
the number of records 
is  then finite (see \cite{Arnold}, Section 6.3). 
In  the exchangeable case the infinitude of records is implied from
 ${\mathbb P}[A_i]=1/i$, and the situation is similar for 
record models with $w(i)=i^\alpha,$   $\alpha\geq -1$ found in Pfeifer \cite{Pfeifer}, as well as for more general regularly varying size functions considered by Steele \cite{Steele}.

The infinitude of records is an intrinsic property, which does not depend on a particular enumeration of items by ${\mathbb Z}_{>0}$, and is equivalent to 
existence of the first element in $\triangleleft$.
The dichotomy begs a deeper question about the {type} of an infinite size-biased order. 
For summable size function the type is ${\mathbb Z}_{>}$, while  in the exchangeable case it is ${\mathbb Q}$.
What are the other possibilities?
Our interest to the question stems from the Arratia, Barbour and Tavar{\'e} conjecture  (\cite{ABT} p. 38, further cited as ABT), which suggests   that
$\triangleleft$ has the type of some suborder in ${\mathbb Z}$
(which may be  ${\mathbb Z}_{<0},   {\mathbb Z}_{>0}$ or ${\mathbb Z}$),  provided
 the multiset of sizes  $(w(i), i\in I)$ has no accumulation points strictly inside $(0,\infty)$. 
A related question concerning existence of few {\it last} elements in the order was addressed in \cite{Chatterjee}.

The rest of the paper is organised as follows. In Section 2 we review constructions of size-biased orders and  their properties. A novelty here is  an insertion algorithm that 
outputs the order for arbitrary size function which need not be summable. In Section 3 we show that the ABT conjecture 
fails without additional conditions
and give a complete classification into possible types. 
Section 4 is devoted to examples.

\section{Constructions and properties}

\subsection{Distribution}

Fix countable $I$ equipped with a size function $w(i)>0$.
We may label $I$ by ${\mathbb Z}_{>0}$ where appropriate, but many properties we are interested in will not depend on such labelling.

The algebraic approach to  size-biased orders relies on  the
 homogeneous  rational functions
\begin{equation}\label{p_n}
p_n(x_1,\dots,x_n)=\prod_{k=1}^n\frac{x_k}{x_k+x_{k+1}+\dots+ x_n}\,\quad 
\end{equation}
of any arity $n\geq 1$.   The symmetrisation of $p_n$  is constant $1$.
Other useful identities is 
 a cycle reversion formula 
$$p_2(x_1,x_2)p_2(x_2,x_3)\cdots p_2(x_{n-1},x_n)=p_2(x_n,x_{n-1})\cdots p_2(x_3,x_2) p_2(x_2,x_1)$$
and a  transposition identity
\begin{eqnarray}\label{quot}
\frac{p_{n+m+2}(a_1\ldots,a_n,x,y, b_1,\ldots,b_m)}{p_{n+m+2}(a_1\ldots,a_n,y,x, b_1,\ldots,b_m)}= \frac{x+(b_1+\cdots+b_m)}{y+(b_1+\cdots+b_m)}.
\end{eqnarray}

The {\it chain} or {\it ranking}  probabilities are obtained by the specialisation of  indeterminates in (\ref{p_n}) as $x_k=w(i_k)$, that is
\begin{equation}\label{chain}
{\mathbb P}[i_1\triangleleft \cdots \triangleleft i_n]=p_n(w(i_1),\ldots,(w(i_n))
\end{equation}
for $\{i_1,\ldots,i_n\}\subset I$. This concluded  from (\ref{exp-r}) using the memoryless property of the exponential distribution.
Permutations of labels yield  probabilities of other chains of relations.

Formula (\ref{chain}) is sometimes adopted as  definition of  the size-biased order on a finite set $J=\{i_1,\ldots,i_n\}$. If $K\supset J$ is another finite set, then the restriction of the size-biased order from $K$ 
to $J$ has distribution (\ref{chain}), which can be verified from the identities like
$$p_2(x,y)=p_3(x,y,z)+p_3(x,z,y)+p_3(z,x,y).$$
In this sense the finite size-biased orders are {\it consistent}.

\paragraph{Intrinsic definition of the order.}
For infinite set, (\ref{chain})  assumes the role of `finite-dimensional distribution'. 
By consistency, the existence of a unique random order satisfying (\ref{chain})  follows by
Kolmogorov's measure extension theorem.
This approach may be regarded as an intrinsic definition of size-biased order, i.e.
not relying on (\ref{exp-r}) or other explicit construction. 
\vskip0.2cm

Many formulas involving $p_n$'s are most elementary shown from the representation (\ref{exp-r}). 
For instance, the recursion
\begin{eqnarray}\nonumber
p_{n+m}(x_1,\dots,x_n, y_1,\ldots,y_m)=& \\
p_{n+1}(x_1,\dots,x_n,  y_1+\cdots +y_m)&\!\!\!\!p_m(y_1,\dots,y_m)
 \label{bas-rec2}
\end{eqnarray}
follows by recognising  in  the left-hand side the ranking probability for 
$n+m$ (independent) exponentials, and in the right-hand side identifying the first factor with the ranking probability for $n+1$ exponentials of which the last is the minimum of $m$ other exponentials.
For $\{i_1,\ldots,i_n\}\cap J=\varnothing, |J|<\infty$ and  $w(J):=\sum_{j\in J}w(J),$ 
we have by (\ref{bas-rec2}) in obvious notation
\begin{equation}\label{head}
{\mathbb P}[i_1\triangleleft\cdots\triangleleft i_n\triangleleft J]=p_{n+1}(w(i_1),\ldots,w(i_n), w(J)).
\end{equation}

Further  features of the size-biased order obvious from (\ref{exp-r}) are
\begin{itemize}
\item[(i)] dissociation: restrictions of $\triangleleft$ to  disjoint $J$ and $K$ are independent, 
\item[(ii)] Luce's property: for $i\in J\subset K$ (finite sets), the probability that $i$ is a size-biased pick from $K$ (i.e. $\triangleleft$-first in $K$) is equal to the product of probability that the size-biased pick 
from $K$ falls in $J$ and the probability that a size-biased pick from $J$ is $i$.
\end{itemize}
The term `dissociation'  derives from the eponymous  property of the bivariate array of indicators $1_{\{i\triangleleft j\}}$,
which means independence of disjoint principal subarrays 
(cf \cite{Kallenberg} p. 339).
Proving the property  directly from (\ref{p_n}) is a nice exercise in algebra, which amounts to showing that
$$
\sum p_{n+m}(z_1,z_2,\ldots,z_{n+m})=p_n(x_1,\ldots,x_n)p_m(y_1,\ldots,y_m),
$$
where the sum is over all `shuffles' (as in Stanley \cite{Stanley} p. 482)
of $x_1,\ldots,x_n$ and $y_1,\ldots,y_m$ with variables from each block appearing in their original succession, for instance $x_1, x_2,y_1, x_3, y_2, y_3,\ldots,x_n$.

Luce's theory starts with a function representing the probability of recognising $i$ as `best' in context $J$; then with the analogue of (ii) postulated, one concludes that
the function is realisable as size-biased pick for some $w$ defines of sizes to items. See \cite{Steele} for a concise derivation.

Continuing the above list:
\begin{itemize}
\item[(iii)] for finite chains $\cdots\triangleleft i \triangleleft j \triangleleft\cdots$ and $\cdots\triangleleft j \triangleleft i \triangleleft\cdots$ that only differ by  transposition of adjacent terms, 
the first has higher probability if $w(i)>w(j)$,
\item[(iv)] for finite $J\subset I$, the most likely ordering is the one which has the items arranged by decreasing size, and the least likely by increasing. 
\end{itemize}
Property (iv) follows from (iii) by induction, and (iii) is clear from (\ref{quot}).

For {\it random} subsets of $I$ the size-biased ordering (permutation) is defined by conditioning.
Let $\tau_1,\ldots, \tau_n$ be a random sequence from $I$ with pairwise distinct elements. We say that the sequence is in size-biased order if
$${\mathbb P}[\tau_1=i_1,\ldots,\tau_n= i_n\,|\,\{\tau_1,\ldots,\tau_n\}=\{i_1,\ldots,i_n\}]={\mathbb P}[i_1\triangleleft\cdots\triangleleft i_n].$$
The property is preserved by  the prefix deletion: 
\begin{itemize}
\item[(v)] if $\tau_1,\ldots, \tau_n$ is in size-biased order, then $\tau_m,\ldots, \tau_n$ is in size-biased order too, for $1\leq m\leq n$,
\end{itemize}
as concluded from (\ref{bas-rec2}).

There is a minor paradox lurking here: 
typically 
a prefix of size-biased sequence is not in size-biased order, so the property fails under {\it suffix} deletion.
To illustrate this, let items $i,j,k$ have sizes $x\neq y, z$ then for their size-biased permutation $\tau_1,\tau_2,\tau_3$
formula (\ref{quot}) yields
$$\frac{{\mathbb P}[\tau_1=i,\tau_2=j]}{{\mathbb P}[\tau_1=j,\tau_2=i]}=\frac{x+z}{y+z}\neq \frac{x}{y}.$$
This has an explanation in terms of (\ref{exp-r}). For independent exponential variables $X_1, X_2, X_3$, given $X_3$ is the largest  the probability  of $X_1<X_2$ depends on the rate of $X_3$.

To extend the deletion theme, let $\tau_1,\ldots,\tau_n$ be a size-biased permutation of $J$. As follows from consistency, if an item is chosen independently
(from $\tau_1,\ldots,\tau_n$)
according to some distribution on $J$ and deleted from 
the permutation, then the resulting sequence is in size-biased order. In the case of uniform choice from $J$ the deletion is equivalent to removing
a term appearing in uniformly random place of the permutation.
Hence deleting $\tau_i$ with  uniformly chosen $i$ yields a size-biased permutation.
The latter implies  a property observed in \cite{PitmanTran} (Corollary 6) for certain random size functions.

\subsection{Construction from a Poisson scatter}

Suppose $I={\mathbb Z}_{>0}$. Independent exponential variables  can be realised as in the Ballerini-Resnick embedding of records in extremal process \cite{Ballerini}.
To that end, consider a unit rate Poisson point process in the domain $D=[0,\,\sum_i w(i))\times [0,\infty)$ of the $(t,x)$-plane.
Split $D$ in semi-infinite strips by vertical lines at points $w(1)+\cdots+w(i)$ and let $(T_i, X_i)$ be the height of the lowest Poisson atom in the $i$th strip.
Clearly, $X_i$ is $w(i)$-rate exponential, $T_i$ uniformly distributed between two division points, and 
all variables $T_i, X_i$ are jointly independent.

This realisation makes obvious many features of $\triangleleft$, in particular that the type of the order is ${\mathbb Z}_{>0}$ if and only if  $\sum_i w(i)<\infty$, in which case the overall lowest point 
in $D$ has the $t$-component uniform on $[0,\,\sum_i w(i))$, hence coinciding with $T_i$ with probability proportional to $w(i)$.

\subsection{Algorithms}

\paragraph{Infinite Lehmer codes.} For a sequence of integers $r_1, r_2,\ldots$ where $r_i$ takes values in $\{1,\ldots,i\}$, there is a unique  order
on ${\mathbb Z}_{>0}$ which places $i$ in position $r_i$ relative to $\{1,\ldots,i\}$. 
For instance $1,2,1,3,\ldots$ means that the order restricted to $\{1,2,3,4\}$ is the permutation $3,1,4,2$. 
The value $r_i=1$ may be interpreted as record in the sense of extreme-value theory,  or that $i$ is the best item among $\{1,\ldots,i\}$ in Luce's theory of choice.  
We shall call $r_i$ the {\it relative rank},
and, by analogy with finite permutations, call
 the whole sequence {\it Lehmer code}, since  $i-r_i$ measures the discordance between $i$ and its relative position in another order.
Note that $i-r_i+1, i\geq 1$ is the Lehmer code of the order reversal.

A random Lehmer code $R_1, R_2,\ldots$ corresponds to a random order on ${\mathbb Z}_{>0}$. 
For the exchangeable order, the relative ranks are independent, with $R_i$ uniform over $\{1,\ldots,i\}$.
For the infinite Mallows order \cite{GO1} with parameter $q>0$, the relative ranks are independent, with $i-R_i$ having truncated geometric distribution on $\{0,1,\ldots,i-1\}$ whose  weights 
are proportional to $1, q,\ldots,q^{i-1}.$

\paragraph{Uniform shuffles.}
Fisher and Yates \cite{Fisher} 
in the introduction to their statistical tables   suggested two algorithms 
to generate a uniform random permutation of $n$ items (Examples 12 and 12.2 in the 4th edition). Both algorithms start with
a source list $S$ with $n$ elements that are moved, one at a time, 
to a target list $T$ where the items  appear shuffled.  The  list $T$ is originally empty.

The algorithm now  known as the Fisher-Yates shuffle iterates the following elemental operation:
\begin{itemize}
\item[(aI)] pick uniformly at random an item from list $S$ and move it to the rear of list $T$.
\end{itemize}

The second algorithm  iterates another operation:  
\begin{itemize}
\item[(aII)] remove the front item from list $S$ and insert it in list $T$ in a uniformly random position. 
\end{itemize}
Note that when $T$ has $i-1$ elements there are $i$ gaps where the next  element from $S$ can be inserted.

Though perhaps less practical computationally, 
the second method has the advantage that it allows one 
to simulate the infinite exchangeable random order in terms of the Lehmer code
$R_1, R_2,\ldots$ be independent, with $R_i$  uniformly distributed on $\{1,\ldots,i\}$.
The variable $R_i$, called {\it relative rank}, is the position occupied by $i$ in the list $T$ when the item is inserted.
 Since there are infinitely many records, the list $T$ will not converge, coordinatewise, to some infinite permutation.

\paragraph{Size-biased shuffles.}
Suppose now that the list $S$ is infinite, with items  having positive sizes whose total is  finite.
The size-biased version of aI is the familiar size-biased pick whose iterates     
yield a size-biased permutation.
Nevertheless, it is instructive to represent a single size-biased pick as a cycle, 
where a pointer driven by a Bernoulli process
moves through  $S$. 
\begin{itemize}
\item[] A cycle begins with the pointer at the front element of $S$.
Each time the pointer is at an item with some size $x$ and let the total of the sizes strictly to the right of the pointer be $t$. With probability $t/(t+x)$ the pointer passes 
to the next item.
Otherwise the current item is placed to the rear of list $T$, and the pointer returns to the front of now reduced $S$ to start the next cycle.
\end{itemize}
Showing that a size-biased pick gets moved relies on
$$
\frac{t_1}{t_1+x_1}\cdots \frac{t_{k-1}}{t_{k-1}+x_{k-1}}\frac{x_k}{t_k+x_k}=\frac{x_k}{t_0},\quad{\rm where~~} t_{i-1}=t_i+x_i, ~i\geq 1.
$$

The analogue  of aII is the operation of size-biased insertion, which we introduce as a cycle of moves.
\begin{itemize}
\item[(aII-sb)] A cycle begins with the pointer in the leftmost gap of $T$.
Let $x$ be the size of the front item in $S$, fixed for the length of the cycle,
 and let $t$ be the variable sum of  sizes in $T$ to the right of the pointer. With probability $t/(t+x)$ the 
pointer moves to the next gap in $T$.
Otherwise the item is 
removed from $S$  and placed in the gap with the pointer, then the pointer returns to the leftmost gap  of $T$ to start the next cycle.
\end{itemize}
\noindent
For the $i$th cycle, if the front item in $S$ has size $x$ the conditional probability of $R_i=k, 1\leq k\leq i,$ given the sizes in $T$ are $x_1,\ldots,x_{i-1}$ is
\begin{equation}\label{quot1}
\left( \prod_{\ell=1}^{k-1} \frac{x_\ell+\dots+x_{i-1}}{x_\ell+\dots+x_{i-1}+x}\right)  \frac{x}{x_k+\dots+x_{i-1}+x}.
\end{equation}
To explain the insertion rule in terms of (\ref{exp-r}), we condition on the ranking event $X_1<\cdots<X_{i-1}$. Given that, the spacings can be represented as
$$X_1=\frac{Y_1}{x_1+\cdots+x_{i-1}}, X_2-X_1=\frac{Y_2}{x_2+\cdots+x_{i-1}}, X_{i-1}-X_{i-2}=\frac{Y_{i-1}}{x_{i-1}},$$
where $x_k=w(k)$ and $Y_1, Y_2,\ldots$ are independent standard exponential variables, see \cite{Nevzorov} p. 19.
Now for $X_i$ another independent exponential  variable with rate $x=w(i)$, the variable falls below $X_1$ with (conditional) probability $x/  (x_1+\cdots+x_{i-1}+x)$, and given this does not happen
falls between $X_1$ and $X_2$ with probability $x/  (x_2+\cdots+x_{i-1}+x)$, and so on.
The case of any other ranking of $X_1,\ldots, X_{i-1}$ is reduced to this by re-labelling.

After each insertion the list $T$ is in size-biased order by the virtue of (\ref{exp-r}). Confirming this algebraically means checking that (\ref{quot1}) coincides with the quotient
$$
\frac{p_i(x_1,\ldots,x_{k-1},x,x_{k+1},\ldots, x_{i-1})}{p_{i-1}(x_1,\ldots,x_{i-1})},
$$
which in turn determines the conditional distribution of $R_i$ given $R_1,\ldots, R_{i-1}$, hence the law of the infinite size-biased order expressed in terms of the relative ranks.

Although the relative ranks are not independent, the events $A_i=\{R_i=1\}$ are independent. 
This property characterises Nevzorov's model for records among a class of order statistics models \cite{Borovkov}.
By the virtue of ${\mathbb P}[A_i]=w(i)/(w(1)+\cdots+w(i))$ the number of records is infinite iff $\sum_i w(i)<\infty$,
iff $\triangleleft$ belongs to the type ${\mathbb Z}_{>0}$.

\paragraph{Time reversal of  the Tsetlin library.}

Consider size-biased version of the familiar  top-to-random shuffle. Let $S$ be a finite list.
In one step, the front element of $S$ is taken and size-biasedly inserted (i.e. using the iterates of (aII-sb)) back into $S$.
It is intuitively clear that the equilibrium of this Markov chain is the size-biased permutation of $S$.

Indeed, if $S$ is a size-biased permutation (of the set of elements of $S$) then removing the front element leaves the other elements in size-biased order,
and the subsequent insertion yields again a size-bised permutation.

In fact, the described process if the time reversal of Tsetlin library, as has been observed by Pitman and Yakubovich \cite{PitmanYak}.
The formula for transition probability of time-reversed Markov chain in equilibrium becomes
the obvious algebraic identity 
\begin{eqnarray*}
\frac{p_{n+1}(x_1,\ldots,x_k,y,x_{k+1},\ldots,x_n)}{p_{n+1}(y, x_1,\ldots,x_n)}\,p_2({y},{y+x_1+\cdots+x_n})=\\
\frac{p_{n+1}(x_1,\ldots,x_k,y,x_{k+1},\ldots,x_n)}{p_{n}(x_1,\ldots,x_n)}.
\end{eqnarray*}

\section{Classification in types}

We procede to the classification of size-biased orders.
The  (noncommutative) operation 
$\nearrow$ will be used to denote the type of an order obtained by putting one set on top of the other 
(the notation adopted in \cite{Rosenstein} is $+$).
 For instance, ${\mathbb Q}\nearrow{\mathbb Z}$ 
is the order type  of  $[{\mathbb Q}\cap (0,1)]\cup \{2^{k}+1: k\in {\mathbb Z}\}$.

Using the realisation (\ref{exp-r}), we may relate the  type of $\triangleleft$ to the topological properties of the random set of points
$\{X_1,X_2,\dots\}$.  Let $\mu$ be the mean measure of the set, which to a finite interval $(x,y)\subset(0,\infty)$,
assigns the mass
\begin{equation}\label{e-ser}
\mu(x,y)=\sum_{i=1}^\infty {\mathbb P}[X_i\in (x,y]] =\sum_{i=1}^\infty (e^{-w(i)x} - e^{-w(i)y}),
\end{equation}
equal to the expected number of exponential points falling in the interval. 
By the Borel-Cantelli lemma we have a dichotomy: the number of points in $(x,y]$ is almost surely finite if $\mu(x,y)<\infty$,
and almost surely infinite if $\mu(x,y)=\infty$.

The measure $\mu$ is absolutely continuous in the sense that it has a finite decreasing density
$$
\varphi(x):=\sum_{i\in I }w(i) e^{-w(i)x}, 
$$
for $x>\beta$, where $\beta\in [0,\infty]$ is the convergence abscissa of the series. 
Note that for finite interval $(x,y)\subset (0,\infty)$
$\mu(x,y)<\infty$ for $x>\beta$, while $\mu(x,y)=\infty$ for $x<\beta$.

In case $\beta=\infty$  the random set $\{X_1, X_2,\ldots\}$ is dense in ${\mathbb R}_{\geq0}$.  Moreover, in this case
by  Tsirelson's universality result \cite{Tsirelson} it is possible to define the $X_j$'s on some probability space together 
with  a sequence of  i.i.d. standard exponential variables 
  $(Y_1, Y_2,\ldots)$ in such a way  that the countable sets $\{X_1, X_2,\ldots\}$ and $\{Y_1, Y_2,\ldots\}$ coincided almost surely.

\paragraph{Case $\sum_i w(i)<\infty$.}
 We have seen already that  $\triangleleft$ is of the type ${\mathbb Z}_{>0}$. This is confirmed by arguing that
since $w(i)\to 0$ and  $e^{-x w(i)} - e^{-y w(i)}\sim w(i)(y-x)$, the series  (\ref{e-ser}) converges everywhere,
$\beta=0$ and $X_i\to\infty$ as $i\to\infty$.

\paragraph{Case  $w(i)\to 0$ and $\sum_i w(i)=\infty$.} The series
  (\ref{e-ser})  diverges everywhere, since
the generic term is asymptotic to $w(i)(y-x)$, hence $\triangleleft$ has the type
${\mathbb Q}$.  This instance disproves the ABT conjecture stated in Introduction.

\paragraph{Case $(w(i), i\in I)$ has accumulation point in $(0,\infty)$.}
The series (\ref{e-ser}) diverges everywhere, because infinitely many terms  $e^{-x w(i)} - e^{-y w(i)}$ are bounded away from $0$. 
Hence $\beta=\infty$,
$\{X_1, X_2,\ldots\}$ is dense in ${\mathbb R}_{>0}$. Thus
 $\triangleleft$ is a dense order of the type ${\mathbb Q}$.

\vskip0.2cm
The  picture is less obvious for  size functions with $w(i)\to\infty$, when large values $X_i$ are rare. To treat this class we need a lemma.
\begin{Lemma}\label{abscissa}
Suppose $w(i)\to\infty$, then $\beta$ is given by 
\begin{equation}\label{beta}
\beta=\limsup_{i\to\infty} \frac{\log i}{w(i)}.
\end{equation}
Thus the  series
{\rm   (\ref{e-ser}) } converges for  $x>\beta$ and diverges for $0\leq x<\beta$ regardless of $y\in (x,\infty]$ (so including $y=\infty$).
\end{Lemma}

\begin{proof} Factoring $e^{-x w(i)} - e^{-y w(i)}=e^{-xw(i)} (1-e^{-(y-x)w(i)})$ we see that 
$\mu(x,y)<\infty$  is equivalent  to  convergence of  the Dirichlet series 
\begin{equation}\label{e-ser-1}
\sum_{i=1}^\infty e^{-x w(i)}.
\end{equation}
Formula (\ref{beta}) is a specialisation of the general formula for the convergence abscissa of the series (cf. \cite{Berenstein} p.489). 
\end{proof}

\paragraph{Case $w(i)\to\infty$ and $0<\beta<\infty$.}
By the lemma, the closure of $\{X_i\}\cap (0,\beta)$ is $[0m\beta]$. Within $(\beta,\infty)$ the sequence $\{X_i\}$ has some finite random number of points
if (\ref{e-ser-1}) converges at $x=\beta$,  or has infinitely many points converging to $\beta$ if (\ref{e-ser-1}) diverges at $x=\beta$.
Accordingly, the type of $\triangleleft$ is either ${\mathbb Q}\nearrow {\mathbb F}$ with ${\mathbb F}$ being a finite ordered set of random cardinality,
or ${\mathbb Q}\nearrow {\mathbb Z}_{<0}$. 
In the former case the type is not pure since ${\mathbb F}$ may have any finite cardinality with positive probability.

\paragraph{Case $w(i)\to\infty$ and  $\beta=0$.}  We have then $X_i \to 0$, hence $\triangleleft$ has the order type ${\mathbb Z}_{< 0}$. 
In this case  $\triangleleft$ is representable as a  left-sided sequence.
The construction by insertion (aII-sb) will produce a list $T$ converging coordinatewise if its elements  are enumerated from the rear to front.

\paragraph{Case $w(i)\to\infty$ and $\beta=\infty$.} Then $\{X_i\}$ is dense in $(0,\infty)$, hence  the order type of $\triangleleft$ is ${\mathbb Q}$.

\vskip0.2cm

It remains to consider the combined cases where the collection of items can be split in two infinite sequences, say $(i_k': w(i_k')\leq 1)$
 and  $(i_k'': w(i_k'')>1 )$, 
such that 
$\sum w(i_k')<\infty$ and $w(i_k'')\to\infty$. 
For the subseries of (\ref{e-ser-1}) taken over $(i_k'') $    the convergence abscissa is $\beta''=\limsup_{k\to\infty} (\log i_k'')/w(i''_k)$ as in Lemma \ref{abscissa}.
Proceeding with these assumptions we distinguish the following possibilities.

\paragraph{Case $0<\beta''<\infty$.} If the  subseries converges at  $x=\beta''$  the type of $\triangleleft$ is ${\mathbb Q}\nearrow {\mathbb Z}_{> 0}$, otherwise it is   ${\mathbb Q}\nearrow {\mathbb Z}$.

\paragraph{Case $\beta''=0$.} Then $\{X_i\}$ has accumulation points $0$ and $\infty$, and $\triangleleft$ is of type $\mathbb Z$.

\paragraph{Case $\beta''=\infty$.} Then  $\triangleleft$ is of type $\mathbb Q$.

\vskip0.2cm
To summarise our findings, we have shown:

\begin{Theorem} The above classification of order types of $\triangleleft$ is complete.
\end{Theorem}

Regarding the ABT conjecture we obtain:

\begin{corollary} The size-biased order $\triangleleft$ can be embedded in ${\mathbb Z}$ if and only if  each of the following three conditions holds:
\begin{itemize}
\item[{\rm(a)}]
$(w(i), i\in I)$ has no accumulation points other than $0$ and $\infty$,
 \item[{\rm(b)}]
$\sum_{\{i: w(i)\leq 1\}} w(i)<\infty$,
\item[{\rm(c)}] if $\infty$ is an accumulation point then $\limsup_{\{i:w(i)>1\}} \frac{\log i}{w(i)}=0.$
\end{itemize}
\end{corollary}

\section{Examples}

Throughout we assume $I={\mathbb Z}_{>0}$.

There are some natural choices for $w$ resulting in the order type  ${\mathbb Z}_{>0}$. 

\paragraph{Karamata-Stirling indicators.} For $\theta>0$ 
\begin{equation}\label{KS}
w(i)=\frac{(\theta)_{i-1}}{(i-1)!}, ~i\in{\mathbb Z}_{>0},
\end{equation}
is the only choice of sizes (with $w(1)=1$) leading to the record indicators with
$${\mathbb P}[A_i]=\frac{\theta}{\theta+i-1},$$
see \cite{GD}.
This profile of success probabilities plays important role in the study of random partitions and other combinatorial structures \cite{CSP}.
The size-biased order type is $\mathbb Q$.

\paragraph{Regular varying size functions.}

Steele \cite{Steele} showed that if the distributions
$$
F_n(t)=\sum_{i\leq nt} \frac{w(i)}{w(1)+\cdots+w(n)}, ~~~t\in [0,1],
$$
weakly converge, then the limit must be a beta distribution $F(t)=t^\theta$ for some $\theta\in[0,\infty]$ (where the limits are $\delta_0,\delta_1$ in the edge cases).
Convergence of $F_n(t_0)$ at a single point $t_0\in(0,1)$ suffices for this.
The limit appears if $w$ is regularly varying as $i\to\infty$ with index $\theta-1$, and if in addition $w$ is monotone then $i\, {\mathbb P}[A_i]\to \theta$.

Recall that for $0<\theta<\infty$ the regular variation amounts to the asymptotics $w(i)\sim i^\theta L(i), ~i\to\infty,$ with some function $L$ of slow variation.
The order $\triangleleft$ is of the type $\mathbb Q$.
Clearly, (\ref{KS}) is a special case.

In the case $\theta=0$ (slow variation) there is some diversity of types, in particular:
\begin{enumerate}
\item For $w(i)=(\log (i+1))^{1/2}$ we have $\beta=\infty$ hence the order type  is ${\mathbb Q}$.
\item For $w(i)=(\log (i+1))^2$ we have $\beta=0$. The order the  type is ${\mathbb Z}_{< 0}$.
 \item For $w(i)=\log i$ the series (\ref{e-ser-1})  is  the Riemann zeta function, diverging for $x\leq1$.
The order  has the combined type ${\mathbb Q}\nearrow{\mathbb Z}_{< 0}$.
\item For $w(i)=\log (i+1)+ 2 \log \log (i+1)$ the series (\ref{e-ser}) converges for $x\geq \beta=1$. The order $\Sigma$ is isomorphic to  
${\mathbb Q}\nearrow  {\mathbb F}$, where $\mathbb F$ is a finite order of random cardinality. 
This example was observed in \cite{Diaconis} as a situations where (in our terms) the order is not of the form $\cdots\nearrow{\mathbb Z}_{<0}$,
that is the reversal of the insertion list $T$ does not converge.
\end{enumerate}

If $w$ varies regularly with index $\theta-1<-1$, the size function is summable, hence the order type is ${\mathbb Z}_{>0}$.

\paragraph{Geometric size function.}
Consider the size-biased order $\triangleleft$ with the geometric size function $w(i)=q^{i}$, $q>0$.
The case  $q>1$ was studied by Yang \cite{Yang} as a model for records in exponentially growing population, proposed to explain the  pattern of Olympic records
that get broken more often than the i.i.d. theory predicts.

The size-biased order shares some features with the Mallows order. Recall that the Mallows order has independent relative ranks with 
$i-R_i$ truncated geometric.
These common features are:
\begin{itemize}
\item[(A)] the restrictions of $\triangleleft$ to $\{1,\ldots,i\}$ and $\{i+1, i+2,\ldots\}$ are independent,
\item[(B)] the restriction of $\triangleleft$ on $\{i+1, i+2,\ldots\}$ under the shift $i+k\mapsto k$ is pushforwarded to a distributional copy of $\triangleleft$, 
\item[(C)] the order type is ${\mathbb Z}_{>0}$ for  $q<1$, and  ${\mathbb Z}_{<0}$ for $q>1$.
\end{itemize}

\noindent
These deserve some comments.  Property (A) for size-biased order is an instance of dissociation,
 which does not hold for 
 Mallows permutation  (as one checks from restrictions on $\{1,3\}$ and $\{2,4\}$).
Property (B) follows from the homogeneity identities like $p_n(q^{i+1},\ldots,q^{i+n})=p_n(1,q,\ldots,q^{n-1})$.
Note also that if $\tau_1,\ldots,\tau_n$ is a size-biased permutation of $\{1,\ldots,n\}$ with parameter $q$ then $n-\tau_1,\ldots,n-\tau_n$ has the same distribution as the size-biased permutation
with parameter $q^{-1}$; the same property is valid for the Mallows permutation.

Suppose $0<q<1$. If $j\triangleleft j$ for $i<j$ we speak on an inversion. Denote $D_n$ the number of inversions within $\{1,\ldots,n\}$, this is `Kendall's tau' measuring the distance of $\triangleleft$ from 
the standard order. In terms of the Lehmer code, $D_n=\sum_{i=1}^n (i-R_i)$.

For items  $i<j$ an inversion occurs with probability  $q^j/(q^j+q^i)$, thus for the expected number of inversions we obtain
$$
{\mathbb E}[D_n]=\sum_{(i,j):1\leq i<j\leq n} \frac{q^j}{q^j+q^i}=\sum_{k=1}^{n-1}\frac{n-k}{1+q^{-k}},
$$
and for $n\to\infty$ the asymptotics becomes
\begin{equation}\label{asymp-inv}
{\mathbb E}[D_n]\sim c_q\,n\,,~{\rm with}~c_q=     \sum_{k=1}^\infty    \left(\frac{1}{1+q^{-k}}\right).
\end{equation}

Moreover, we even have the asymptotics for $D_n$ in a strong sense:
$$D_n\sim c_q n~~~{\rm a.s.}$$
Indeed, for two disjoint integer intervals $[a,b]$ and $[c,d]$ with $b-a=d-c$, the number of inversions of the restrictions of $\triangleleft$ are independent and have same distributions as $D_{b-a}$ and $D_{d-c}$,
respectively, by properties (A) and (B). 
On the other hand, the number of inversions for union of two disjoint sets in not smaller than the sum of the inversion counts for each of the sets.
With account of the already established asymptotics of the mean (\ref{asymp-inv}) the result follows now by applying  the i.i.d. version of Kingman's subadditive ergodic theorem (see, e.g., \cite{Romik}, Theorem A.3).

To compare, the number of inversions on $\{1,\ldots,n\}$ for the Mallows permutation is asymptotic to $(q^{-1}-1)^{-1} n$, as follows from the distribution of relative ranks.

\paragraph{Acknowledgement.}
I am indebted to Jim Pitman for useful comments and drawing my attention to paper \cite{PitmanYak}.

\end{document}